\documentclass[a4paper,11pt]{article}
\usepackage{amssymb}
\usepackage{enumerate,verbatim}
\usepackage{amsmath}
\usepackage{amsfonts}
\usepackage{amsthm}
\usepackage{xcolor}
\usepackage{epsfig}
\usepackage[normalem]{ulem}

\newtheorem{theorem}{Theorem}
\newtheorem{remark}{Remark}
\newtheorem{definition}{Definition}

\newtheorem{corollary}{Corollary}
\newtheorem{lemma}{Lemma}
\newtheorem{proposition}{Proposition}

\newcommand{\C}{\mathbb{C}}

\def\eqref#1{(\ref{#1})}
\begin{document}
\author{
Colin Christopher\\
School of Engineering, Computing and Mathematics\\Plymouth University\\
Plymouth, PL4 8AA, UK\\
\tt{C.Christopher@plymouth.ac.uk}\\
{}\\
Sebastian Walcher\\
Fachgruppe Mathematik, RWTH Aachen\\
52056 Aachen, Germany\\
\tt{walcher@mathga.rwth-aachen.de}
}
\title{Exceptional rational vector fields with an elementary first integral}
\maketitle

\begin{abstract}
We consider complex rational vector fields that admit a first integral whose logarithmic derivative lies in a finite extension of the  rational function field $K$. In view of the Prelle-Singer theorem, these are the rational vector fields that admit an elementary first integral. Elementary integrable vector fields which are not Darboux integrable --- thus the extension field is necessarily a proper extension of $K$ --- may be called exceptional by an observation in an earlier paper by Christopher et al. For dimension two we characterize all possible algebraic extension fields underlying the exceptional cases, provide a construction of all exceptional vector fields, and obtain some criteria that restrict the degree of $L$.\\
MSC (2020): 34A05, 12H05, 34M15\\
Key words: Elementary integrability, Darboux integrability, cyclic field extensions.

\end{abstract}

\section{Introduction}

In the present paper we discuss rational complex vector fields that admit a first integral in some elementary extension of the rational function field over $\mathbb C$.\footnote{Throughout, the field $\mathbb C$ may be replaced by any algebraically closed field of characteristic zero.} \\
Thus we will consider vector fields
\begin{equation}\label{vfield3d}
\mathcal{X}=\sum_{i=1}^nP_i\frac{\partial}{\partial x_i}
\end{equation}
on $\mathbb C^n$ with rational $P_i$; equivalently the corresponding ($n-1$)-forms
\begin{equation}\label{2form}
\Omega=\sum_{i=1}^nP_i\,dx_1\wedge\cdots\wedge\widehat{dx_i}\wedge\cdots dx_n
\end{equation}
defined over $K:=\mathbb C(x_1,\ldots,x_n)$.\\
We recall some definitions and facts (see e.g. Rosenlicht \cite{Ros}).
\begin{definition}\label{elemdef}{\em
\begin{enumerate}[(a)]
\item A differential extension field $M$ of $K$ is {\it elementary} if and only if $K$ and $M$ have the same constants
 and there exists a tower of fields
\begin{equation}\label{LiouvExeq} K=K_0 \subset K_1 \subset \ldots \subset K_m=M, \end{equation}
such that for each $i\in\{0,\ldots,m-1\}$ we have one of the following:
\begin{enumerate}[(i)]
\item $K_{i+1}=K_i(t_i)$, where $t_i\not=0$ and $dt_i/t_i=dR_i $ with some $R_i \in K_i$ (adjoining an exponential);
\item $K_{i+1}=K_i(t_i)$, where $dt_i=dR_i/R_i$ with $R_i \in K_i$ (adjoining a logarithm);
\item $K_{i+1}$ is a finite algebraic extension of $K_i$.
\end{enumerate}
\item A non-constant
element, $\phi$, of an elementary extension of $K$ is called an {\it elementary first integral}  of the vector field $\mathcal{X}$ if it satisfies $\mathcal{X} \phi=0$ or, equivalently, $d\phi \wedge  \Omega=0$.
\end{enumerate}
}
\end{definition}
The condition on the constants is unproblematic in our context; see e.g. \cite{ACPW}, Remark 1.

An elementary first integral exists if and only if there exist an elementary extension $M$ of $K$ and $\widetilde v\in M$, $\widetilde u_1,\ldots,\widetilde u_r\in L^*$, and $\widetilde c_1,\ldots,\widetilde c_r\in \mathbb C$ such that
\begin{equation}\label{elemrelcore}
    \left(\sum_{i=1}^r \widetilde c_i\,\dfrac{d\widetilde u_i}{\widetilde u_i} +d\widetilde v\right)\wedge\Omega=0, \quad \sum_{i=1}^r \widetilde c_i\,\dfrac{d\widetilde u_i}{\widetilde u_i} +d\widetilde v\not=0.
\end{equation}
To verify the non-obvious implication, adjoin logarithms if needed. One may choose the $\widetilde c_i$ to be linearly independent over the rationals $\mathbb Q$.

Our vantage point is the following version of the main theorem in Prelle and Singer \cite{PreSin}, which we state for rational ($n-1$)-forms\footnote{Huang and Zhang \cite{HZ} recently extended the Prelle-Singer theorem to $(n-1)$-forms admitting $n-1$ independent elementary first integrals.}.
\begin{theorem}\label{algaddendum}
Let $K = \C(x_1,\ldots, x_n)$, and $\Omega$ the ($n-1$)--form {\eqref{2form}} over $K$. If there exists an elementary extension $M$ of $K$ such that a relation \eqref{elemrelcore} holds with $\widetilde u_i,\,\widetilde v\in L$, then there exists a finite algebraic extension $L$ of $K$, $L\subseteq M$,  such that a relation 
\begin{equation}\label{elemrel}
    \left(\sum_{i=1}^m c_i\,\dfrac{d u_i}{ u_i} +d v\right)\wedge\Omega=0, \quad \sum_{i=1}^m  c_i\,\dfrac{d u_i}{ u_i} +d v\not=0.
\end{equation}
holds with $ c_i\in \mathbb C$, and $u_i,\, v\in \widetilde L$.
\end{theorem}
By \cite{CLPW}, Theorem 2 (which we recall below), elementary integrability implies Darboux integrability --- thus one may take  $L=K$ --- unless all $\widetilde u_i$ have constant norm and $\widetilde v$ has constant trace. We call vector fields of this type exceptional. We will discuss these exceptional elementary integrable vector fields in dimension two, continuing \cite{CLPW}.
The principal results are as follows:
\begin{itemize}
\item We show that the underlying (minimal) field extensions are cyclic extensions of $K$;
\item we  obtain a construction of all exceptional vector fields; and
\item we provide criteria which guarantee  that the degree of the field extension equals two, including  a complete characterization when relation \eqref{elemrel} holds with $m=1$.
\end{itemize}


\section{Darboux vs. exceptional cases}
From here on we assume that $L$ is algebraic over $K$, and minimal. The proof of Theorem 2(a) in \cite{CLPW}, although written for the setting with dimension two, carries over directly to show that vector fields admitting a relation \eqref{elemrel} are frequently Darboux integrable. In precise terms:
\begin{proposition}
    Let $L$ be a finite algebraic extension of $K$, and let $v\in L$, $u_1,\ldots,u_m\in L^*$ and $c_1,\ldots,c_m\in \mathbb C$ such that \eqref{elemrel} holds.
Then, whenever $v$ has non-constant trace, or some $u_i$ has non-constant norm, there exist $\widehat v\in K$ and $\widehat u_1,\ldots, \widehat u_{\widehat m} \in K^*$ so that 
\begin{equation*}
    \left(\sum_{i=1}^{\widehat m} c_i\,\dfrac{d\widehat u_i}{\widehat u_i} +d\widehat v\right)\wedge\Omega=0, \quad \sum_{i=1}^{\widehat m} c_i\,\dfrac{d\widehat u_i}{\widehat u_i} +d\widehat v\not=0.
\end{equation*}
In particular, $\Omega$ admits a Darboux first integral.\footnote{Note that the class of Darboux first integrals includes rational first integrals. An explicit description of Darboux integrable vector fields in dimension two is provided in Chavarriga et al.~\cite{CGGL}.}
\end{proposition}

In view of this result we call rational vector fields with an elementary  first integral {\em exceptional} if $v$ has constant trace and all $u_i$ have constant norm; thus necessarily $L\not=K$. When discussing exceptional cases we may obviously assume trace zero and norm one, respectively.
\section{Exceptional cases in dimension two}
From here on we restrict attention to exceptional vector fields (or one-forms) in dimension two. We will generalize and improve results from \cite{CLPW}.\\ 
Changing notation, we let $K=\mathbb C(x,y)$, and consider a one-form
\[
\omega=P\,dx+Q\,dy,
\]
with corresponding vector field $\mathcal X=-Q\,\partial/\partial x +P\,\partial/\partial y$, that admits an elementary first integral, thus 
\begin{equation}\label{elem2d}
\left(\sum_{i=1}^m c_i\dfrac{du_i}{u_i}+dv\right)\wedge\omega=0,
\end{equation}
with $v$ and the $u_i$ in some finite algebraic extension $M$ of $K$, and complex constants $c_i$. One may restate \eqref{elem2d} as
\begin{equation}\label{elem2dfunc}
   d\psi\wedge\omega=0 ,\text{  with  }\psi:= \sum_{i=1}^mc_i\,\log u_i \, +v.
\end{equation}
We recall and reprove a familiar fact about integrating factors; see e.g. Prelle and Singer \cite{PreSin}, Proposition 1.
\begin{lemma} There exists $\ell\in M$ such that
\begin{equation}\label{2dint}
    \sum c_i\dfrac{du_i}{u_i}+dv=\ell\omega; \quad d\omega=-\dfrac{d\ell}{\ell}\wedge\omega.
\end{equation}
\end{lemma}
\begin{proof}
    The first statement holds because the one-forms $\beta$ with $\beta\wedge\omega=0$ form a vector space of dimension one over $M$; the second follows from $d(\ell\,\omega)=0$.
\end{proof}
\begin{proposition}\label{cyclo}Assume that $\omega$ is exceptional, and let $L\subseteq M$ be a minimal algebraic extension of $K$ such that \eqref{2dint} holds with $v$ and all $u_i$ in $L$. Then:
\begin{enumerate}[(a)]
    \item $L=K(\ell)$ is Galois over $K$, with cyclic Galois group $G\cong \mathbb Z_n$.
    \item Given a generator $\tau$ of $G$, there is a primitive $n^{\rm th}$ root of unity, $\zeta$, such that $\tau(\ell)=\zeta\,\ell$, and moreover $\ell=\sqrt[n]{k}$ for some $k\in K$.
\end{enumerate}
\end{proposition}
\begin{proof}\begin{enumerate}[(i)]
    \item We may assume that $M$ is Galois over $K$, with Galois group $\widetilde G$. Assume that $\sigma(\ell)/\ell$ is not constant for some $\sigma\in\widetilde G$. Then \eqref{2dint} implies
    \[
    d\omega=-\dfrac{d\sigma(\ell)}{\sigma(\ell)}\wedge\omega,
    \]
    and
    \[
    d\left(\dfrac{\sigma(\ell)}{\ell}\right)\wedge\omega=\dfrac{1}{\ell^2}\left(\ell\,d\sigma(\ell)-\sigma(\ell)\,d\ell\right)\wedge\omega=0.
    \]
    Thus, if $\sigma(\ell)/\ell$ is not constant for some $\sigma$, then $\omega$ admits an algebraic first integral, hence a rational first integral (see e.g. \cite{CLPW}, Lemma 2). 
    \item Since $\omega$ represents an exceptional case, there exists for every $\sigma\in \widetilde G$ some $\lambda_\sigma\in\mathbb C^*$ such that
    \[
    \sigma(\ell)=\lambda_\sigma\cdot\ell,
    \]
    and one sees that
    \[
    \widetilde G\to\mathbb C^*,\quad \sigma\mapsto\lambda_\sigma
    \]
    is a homomorphism of groups. Let $\widetilde G_0$ be the kernel of this homomorphism. Then $\widetilde G/\widetilde G_0\cong \mathbb Z_n$ for some $n>1$, since all finite subgroups of $\mathbb C^*$ are cyclic.
    \item Let $L$ be the fixed field of $\widetilde G_0$. Then $L$ is a Galois extension of $K$, with Galois group $G\cong\widetilde G/\widetilde G_0$. Moreover \eqref{elem2d} implies
    \[
    \begin{array}{rcl}
    \sum c_i\left(\dfrac{d\prod_{\sigma\in\widetilde G_0}\sigma(u_i)}{\prod_{\sigma\in\widetilde G_0}\sigma(u_i)}\right)+\sum_{\sigma\in\widetilde G_0}d\sigma(v) &=&\\
        \sum c_i\left(\sum_{\sigma\in\widetilde G_0}\dfrac{d\sigma(u_i)}{\sigma(u_i)}\right)+\sum_{\sigma\in\widetilde G_0}d\sigma(v) &=&\\
        \sum_{\sigma\in\widetilde G_0}\left(\sum c_i\dfrac{d\sigma(u_i)}{\sigma(u_i)}+dv\right)&=&|\widetilde G_0|\cdot\ell\, \omega,
    \end{array}
    \]
    hence we may assume that \eqref{elem2d} holds with $v$ and all $u_i$ in $L$. 
    \item Let $\tau$ be a generator of $G$. Then $\lambda_\tau=\zeta$ is a primitive $n^{\rm th}$ root of unity. Moreover $\tau(\ell^n) = \zeta^n\ell^n=\ell^n$, hence $k:=\ell^n\in K$. Since $K(\ell)$ is normal (with all roots of unity in $K$), minimality shows $L=K(\ell)=K(\sqrt[n]{k})$.
\end{enumerate}
\end{proof}

\begin{remark}{\em 
     Proposition 2 in Prelle and Singer \cite{PreSin} shows that every elementary integrable vector field in dimension two admits an integrating factor $\widehat \ell$, such that $\widehat \ell^{n^*}\in K$ for some $n^*$. In the exceptional case we see that the integrating factor generates the extension field $L$ over $K$.}
\end{remark}
In the following subsections we consider the construction of exceptional vector fields, and conditions for their existence, given the degree $[L:K]$ of the cyclic field extension. We will discuss degree two separately, providing a detailed description of the exceptional vector fields. 
\subsection{Quadratic extensions}
For degree two field extensions, we have the setting of Proposition \ref{cyclo}, with $\ell^2=k\in K$. In the case  $m=1$ (with just one logarithmic term) in \eqref{elem2d}, all two dimensional exceptional vector fields were determined in \cite{CLPW}, Theorem 3. Here we consider the case of general $m\geq 1$, from a different vantage point, and obtain an explicit description of all exceptional vector fields.

From $\ell^2=k$ we obtain
\[
2\ell\ell_x=k_x;\quad \ell_x=\dfrac{k_x}{2k}\cdot \ell;
\]
and similarly for partial derivatives with respect to $y$.
A trace zero element $v\in L$ is necessarily of the form $v=h\cdot \ell,\, h\in K$;  so  
\[
v_x=h_x\ell+h\ell_x=\left(h_x+\frac12 h\dfrac{k_x}{k}\right)\,\ell;
\]
and similarly for $v_y$.\\
$L$ admits a nontrivial automorphism $\tau$ which sends $\ell$ to $-\ell$. By Hilbert's Theorem 90 (see e.g. Lang \cite{SergeL}, Ch.~VI, Thm.~6.1) norm one elements are of the form
\[
u=\dfrac{\tau(w)}{w};\quad w=g+f\ell\in L^*,
\]
explicitly
\[
u=\dfrac{g^2+kf^2}{g^2-kf^2}+\dfrac{2gf}{g^2-kf^2}\ell.
\]
For nonconstant $u$ one has $f\not=0$, thus one may choose $f=1$, and 
\begin{equation}\label{gexpr}
    u=\dfrac{g^2+k}{g^2-k}+\dfrac{2g}{g^2-k}\ell=:a+b\ell,
\end{equation}
with
\begin{equation}\label{normonestuff}
    a^2-kb^2=1,\quad u^{-1}=a- b\ell.
\end{equation}
Differentiation of \eqref{gexpr} yields
\[
u_x=a_x+b_x\ell +b\ell_x=a_x+\left(b_x+\frac12b\dfrac{k_x}{k}\right)\,\ell;
\]
\[
\begin{array}{rcl}
  u_x/u&=   & aa_x-kb\left(b_x+\frac12b\dfrac{k_x}{k} \right)\\
     & &+\ell\cdot\left(-a_xb+ab_x+\frac12ab\dfrac{k_x}{k}\right).
\end{array}
\]
Differentiating the first identity in \eqref{normonestuff} with respect to $x$, one sees that the first term on the right hand side vanishes. After rearranging one gets
\[
\dfrac{u_x}{u}=ab\ell\,\left(\dfrac{(b/a)_x}{(b/a)}+\frac12\dfrac{k_x}{k}\right),
\]
and a similar expression holds for $u_y/u$.\\
It is straightforward to rewrite these expressions in terms of $g$. Now the elementary first integral, $\psi$, is the sum of some trace zero element $v$ and a $\mathbb C$-linear combination of logarithms of norm one elements $u_i$, which are, in turn, determined by certain rational $g_i$, according to \eqref{gexpr}. The final result is as follows:
\begin{proposition}\label{quadprop}
\begin{enumerate}[(a)]
    \item 
    Every exceptional vector field in dimension two is a rational multiple of 
    \[
    \begin{array}{ll}
       H:=  & \begin{pmatrix}
             -(h_y+\frac12h\dfrac{k_y}{k})\\
             h_x+\frac12h\dfrac{k_x}{k}
         \end{pmatrix} \\
         +\sum c_i\dfrac{1}{(g_i^2-k)^2}& \begin{pmatrix}
             -2\left(-g_i^2g_{i,y}+kg_{i,y}-g_ik_y\right)-g_i(g_i^2+k)\dfrac{k_y}{k}\\
             2\left(-g_i^2g_{i,x}+kg_{i,x}-g_ik_x\right)+g_i(g_i^2+k)\dfrac{k_x}{k}
         \end{pmatrix},
    \end{array}
    \]
    with rational functions $h$ and $g_i$.
    \item A vector field $H$ of the form in part (a) is exceptional if and only if it admits no rational first integral.
    \end{enumerate}
\end{proposition}
\begin{proof} Part (a) is obtained from the Hamiltonian of $\psi$. As to part (b), 
    $H$ admits the integrating factor $\ell^{-1}$ by construction. If $H$ is not exceptional then it admits a Darboux first integral, hence a rational integrating factor by Chavarriga et al. \cite{CGGL}. The quotient of the two integrating factors is  nonconstant and an algebraic first integral. The existence of a rational first integral follows.
\end{proof}
Thus, to construct all exceptional vector fields, one may start with arbitrary rational functions and apply Proposition \ref{quadprop}. But there is extra work involved to preclude the existence of rational first integrals.
\subsection{Field extensions of degree greater than two}
We now consider the setting (keeping the notation) of Proposition \ref{cyclo}, with $[L:K]=n>2$, and present a general construction that yields all exceptional vector fields.\footnote{The construction is also applicable for degree two, but Proposition \ref{quadprop} provides greater detail.}\\
Denote by $K'$, resp. $L'$, the space of one-forms over $K$, resp. $L$. Then every $\beta \in L'$ admits a unique representation 
\begin{equation}\label{betaLK}
    \beta=\sum_{i=0}^{n-1}\ell^i\,\beta_i;\quad \beta_i\in K'.
\end{equation}
\begin{lemma}
    The map
    \[
    {\rm Pr}:\, L'\to L',\quad \xi\mapsto\frac1n\sum_{j=0}^{n-1}\zeta^{-j}\tau^j(\xi)
    \]
    is a $K$-linear projection from $L'$ onto $\ell\,K'$. In particular, $\beta$ as given in \eqref{betaLK} is mapped to $\ell\beta_1$.
\end{lemma}
\begin{proof}
    It suffices to verify the last assertion. For any $i$ one has
    \[
    \begin{array}{rcl}
      {\rm Pr}(\ell^i\beta_i)&=   & \frac1n\sum_j\zeta^{-j}\tau^j(\ell^i) \\
         & =&\frac1n \sum_j\zeta^{(i-1)j}\ell^i\beta_i,
    \end{array}
    \]
    and $\sum_{j=0}^{n-1}\zeta^{(i-1)j}=0$ whenever $i\not=1$.
\end{proof}
This yields a general construction principle for exceptional vector fields.
\begin{proposition}
\begin{enumerate}[(a)]
    \item  Let
    \begin{equation*}
       \widetilde \gamma=d\widetilde v+\sum \widetilde c_i\dfrac{d\widetilde u_i}{\widetilde u_i}\in L',
    \end{equation*}
    with $\widetilde v\in L$ of trace zero, all $\widetilde u_i\in L^*$ of norm one, and $\widetilde c_i\in \mathbb C$. Then \[\gamma:={\rm Pr}\,(\widetilde\gamma)=dv+\sum c_i\dfrac{du_i}{u_i}
    \]
    with $v\in L$ of trace zero, all $u_i\in L^*$ of norm one, and $c_i\in \mathbb C$. Furthermore
    \[
    \gamma=\ell\,\omega\text{  for some  }\omega\in K'.
    \]
    \item If $\omega$ does not admit a rational first integral, then $\omega$ is exceptional.
    \item Every exceptional form in $L'$ is obtained as an image under ${\rm Pr}$.
\end{enumerate}
\end{proposition}
\begin{proof} As to part (a), automorphisms preserve traces and norms, and the image of ${\rm Pr}$ is $\ell\,K'$.
    For part (b), note that $\ell$ is an integrating factor for $\omega$, and repeat the argument in the proof of Proposition \ref{quadprop}. Part (c) is immediate since ${\rm Pr}$ is a projection.
\end{proof}
\subsection{Degree restrictions}
Finally we consider elementary first integrals \eqref{elem2d}, and discuss possible degrees of the cyclic extension fields involved. The following general result indicates the relevance of the constants $c_i$.
\begin{proposition}\label{czeta}
    Assume that $\omega$ is exceptional, and let $L$ be a minimal algebraic extension of $K$, with $[L:K]=n>1$, such that \eqref{elem2d} holds with $v$ and $u_i$ in $L$, and $c_1,\ldots,c_m$ linearly independent over $\mathbb Q$.
    If $\zeta$ is a primitive $n^{\rm th}$ root of unity, then $c_1,\ldots,c_m,\zeta c_1,\ldots,\zeta c_m$ are linearly dependent over $\mathbb Q$.
\end{proposition}

The proof of Proposition \ref{czeta} is a generalization of the proof of Theorem 4 in \cite{CLPW}, and is based on a theorem by Rosenlicht \cite{RosIHES}. The following auxiliary result extends Lemma 5 in \cite{CLPW}.
\begin{lemma}\label{algclose}
    Let $L=K(\ell)$ as in Proposition \ref{cyclo}, and $k=\ell^n\in K$. Moreover, write
    \begin{equation}\label{exponents}
    k=q_1^{e_1}\cdots q_r^{e_r}
    \end{equation}
    with relatively prime irreducible polynomials $q_i\in\mathbb C[x,\,y]$, and nonzero integers $e_i$. If no $q_i$ lies in $\mathbb C[x]$, then $\mathbb C(x)$ is algebraically closed in $L$.
\end{lemma}
\begin{proof}
We may assume that \eqref{exponents} holds with $0<|e_i|<n$ for all $i$, since any exponent of absolute value $\geq n$ for some $q_i$ implies $\ell=q_i\cdot\widehat\ell$, with $L=K(\ell)=K(\widehat \ell)$.\\
Now let $R\in L$ be algebraic over $\mathbb C(x)$. We have to show that $R\in\mathbb C(x)$.
\begin{enumerate}[(i)]
    \item With the minimal polynomial of $R$ over $\mathbb C(x)$ we get
    \[
    R^p+\sum_{i=1}^{p}a_i R^{p-i}=0,\quad\text{ with  }a_i\in \mathbb C(x).
    \] 
    Differentiation with respect to $y$ shows
    \[
    R_y\,\left(pR^{p-1}+\sum (p-i)a_i\,R^{p-i-1}\right)=0,
    \]
     and the second factor is nonzero, therefore $R_y=0$.
     \item From the representation
     \[
     R=\sum_{i=0}^{n-1}b_i\ell^i,\quad b_i\in K,
     \]
      and recalling $\ell_y/\ell=\frac1n k_y/k$,  we find
     \[
     0=R_y=b_{0,y}+\sum_{i=1}^{n-1}\left(b_i,y+\frac in\dfrac{k_y}{k}b_i\right)\ell^i.
     \]
     Therefore $b_0\in \mathbb C(x)$, and 
     \[
     n\cdot\dfrac{b_{i,y}}{b_i}+i\cdot\dfrac{k_y}{k}=0\text{  if  }b_i\not=0,\quad 1\leq i \leq n-1,
     \]
     which implies 
     \[
     b_i^n\cdot k^i \in \mathbb C(x),\quad 1\leq i\leq n-1.
     \]
     Now assume $b_i\not=0$. Comparing prime factors in numerators and denominators of the terms in this expression yields
     \[
     b_i=q_1^{m_{i,1}}\cdots q_r^{m_{i,r}}\cdot \widetilde b_i,\quad\text{with  } \widetilde b_i\in\mathbb C(x),
     \]
     and
     \[
     n\cdot m_{i,j}+i\cdot e_j=0.
     \]
     \item For $i=1$ this relation implies $n|e_j$; a contradiction to $|e_j| <n$. So $b_1=0$. For $i>1$ the relation implies that $d:={\rm gcd}\,(n,i)>1$, and therefore $e_j=d\cdot\widetilde e_j$ with integers $\widetilde e_j$.
     So 
     \[
     k=\left(q_1^{\widetilde e_1}\cdots q_r^{\widetilde e_r}\right)^d,
     \]
     and 
     \[
     \ell^{n/d}=q_1^{\widetilde e_1}\cdots q_r^{\widetilde e_r}\in K.
     \]
     But this implies $[L:K]=[K(\ell):K]\leq n/d$; a contradiction. \\
     To summarize, we have $b_1=\cdots=b_{n-1}=0$, and $R=b_0\in \mathbb C(x)$.
\end{enumerate}
\end{proof}
\begin{proof}[Proof of Proposition \ref{czeta}]
    \begin{enumerate}[(i)]
        \item By a $\mathbb C$-linear transformation $\begin{pmatrix}
            x\\y
        \end{pmatrix}\mapsto A\begin{pmatrix}
            x\\y
        \end{pmatrix}$
         we may assume that no prime factor in the numerator or the denominator of $k$ lies in $\mathbb C(x)$, hence $C(x)$ is algebraically closed in $L$ by Lemma \ref{algclose}.
         \item We have 
         \[
    \sum c_i\dfrac{du_i}{u_i}+dv=\ell\omega
    \]
from \eqref{elem2d}, hence 
\[
\sum c_i\dfrac{u_{i,x}}{u_i}+v_x=-\ell Q.
\]
With the $K$-automorphism $\sigma$ that sends $\ell$ to $\zeta^{-1}\ell$, one finds
\[
\sigma(v_x)+\sum c_i\dfrac{\sigma(u_{i,x})}{\sigma(u_i)}=\zeta^{-1}\ell Q,
\]
and furthermore
\begin{equation}\label{roseneq}
\left(v-\zeta\sigma(v)\right)_x-\sum\zeta c_i\dfrac{\sigma(u_i)_x}{\sigma(u_i)}+\sum c_i\dfrac{u_{i,x}}{u_i}=0.
\end{equation}
\item Now assume that $c_1,\ldots, c_m,\zeta c_1,\ldots,\zeta c_m$ are linearly independent over $\mathbb Q$. Then the theorem in Rosenlicht \cite{RosIHES} is applicable (use this result with $k\to \mathbb C(x)$, $K\to L$, $t \to y$ and $'=\frac{\partial}{\partial x}$), due to Lemma \ref{algclose}. By this theorem, all $u_i$ are in $\mathbb C(x)$, and $w:=v-\zeta\sigma(v)$ satisfies
\[
w=\gamma y+\widetilde w,\quad \text{with  }\gamma\in\mathbb C,\,\widetilde w\in \mathbb C(x).
\]
\item For $v$, we have
\[
v=\sum_{i=0}^{n-1}a_i\ell^i,\quad\text{ with  }a_i\in \mathbb C(x,y),
\]
and
\[
v-\zeta\sigma(v)=\sum_{i=0}^{n-1}(1-\zeta^{1-i})a_i\ell^i.
\]
Comparing this with the expression for $w$ in (iii), one sees that
\[
a_0=(1-\zeta^{-1})\left(\gamma y+\widetilde w\right)\in \mathbb C(x,y);\quad a_i=0\text{  for all  }i\geq 2,
\]
hence $v=a_0+a_1\ell$. Moreover
\[
dv= da_0+\left(da_1+\frac 1n\frac{dk}{k}a_1\right)\ell.
\]
\item Now, comparing expansions in powers of $\ell$ in \eqref{2dint} shows, in view of (iii), that 
\[
\omega = da_1+\frac 1n\frac{dk}{k}a_1,
\]
and 
\[
nk\omega=d\left(a_1^n\cdot k\right).
\]
By this equality, $\omega$ admits a rational first integral. Thus the assumption of linear independence of $c_1,\ldots,c_m,\zeta c_1,\ldots, \zeta c_m$ yields a contradiction. 
    \end{enumerate}
\end{proof}
The following corollary shows that only degree two field extensions are possible for many choices of the constants $c_i$. Moreover it includes a generalization of \cite{CLPW}, Theorem 4 (which in \cite{CLPW} was proven only for prime degree).
\begin{corollary} Assume that $\omega$ is exceptional, and that \eqref{elem2d} holds. 
\begin{enumerate}[(a)]
    \item If $c_1,\ldots,c_m,\zeta c_1,\ldots,\zeta c_m$ are linearly independent over $\mathbb Q$ for all primitive $n^{\rm th}$ roots of unity, $n\geq 3$, then $[L:K]=2$.
    \item In particular, whenever $c_1,\ldots,c_m$ are linearly independent over the algebraic closure $\overline{\mathbb Q}$ of $\mathbb Q$  in $\mathbb C$, then $[L:K]=2$.
    \item  In particular, whenever \eqref{elem2d} holds with $m=1$, thus an exceptional vector field admits a first integral of the form $\log u+v$, then $[L:K]=2$.
\end{enumerate}

\end{corollary}

\end{document}